\documentclass[11pt,final]{iopart}
\usepackage{amssymb}
\usepackage{amsfonts}
\usepackage{amsthm}
\usepackage{iopams}
\usepackage[dvips]{epsfig}
\usepackage{float}
\usepackage{footmisc}
\usepackage[title,titletoc,toc]{appendix}
\usepackage{prettyref}
\usepackage{appendix}
\usepackage{showkeys}
\usepackage{todo}
\usepackage{prettyref}
\usepackage{verbatim}
\usepackage[latin1]{inputenc}
\usepackage{subfigure}
\usepackage{verbatim}
\newrefformat{eq}{(\ref{#1})}
\newrefformat{fig}{Figure~\ref{#1}}
\usepackage[pdfpagelabels]{hyperref}
\hypersetup{colorlinks=true,
linkcolor=blue,
citecolor=blue,
plainpages=false}
\newtheorem{theorem}{\bf Theorem}[section]
\newtheorem{propn}[theorem]{\bf Proposition}
\newtheorem{assump}[theorem]{\bf Assumption}

\newtheorem{lemma}[theorem]{\bf Lemma}

\newtheorem{definition}[theorem]{\bf Definition}
\newtheorem{remark}[theorem]{\bf Remark}

\newtheorem{cor}[theorem]{\bf Corollary}

\newcommand{\argmin}{{\rm arg\,min}}

\eqnobysec

\makeindex

\newcommand{\cB}{{\mathcal B}}
\newcommand{\cC}{{\mathcal C}}
\newcommand{\cD}{{\mathcal D}}

\newcommand{\cH}{{\mathcal H}}

\newcommand{\cL}{{\mathcal L}}

\newcommand{\cP}{{\mathcal P}}

\newcommand{\cT}{{\mathcal T}}

\newcommand{\cV}{{\mathcal V}}
\newcommand{\cW}{{\mathcal W}}

\newcommand{\cZ}{{\mathcal Z}}

\newcommand{\N}{\mathbb{N}}
\newcommand{\R}{\mathbb{R}}

\newcommand{\norm}[1]{|| #1||}

\renewcommand{\div}{{\rm div \;}}

\newcommand{\ba}{\begin{array}}
\newcommand{\ea}{\end{array}}
\newcommand{\be}{\begin{equation}}
\newcommand{\ee}{\end{equation}}
\newcommand{\bea}{\begin{eqnarray}}
\newcommand{\eea}{\end{eqnarray}}
\newcommand{\beq}{\begin{equation}}
\newcommand{\eeq}{\end{equation}}
\newcommand{\bqt}{\begin{quote}}
\newcommand{\eqt}{\end{quote}}

\begin{document}

%
\title[Variational Convergence Analysis]
{Variational Convergence Analysis With Smoothed-TV Interpretation}

%
\author{Erdem Altuntac}

\address{Institute for Numerical and Applied Mathematics,
University of G\"{o}ttingen, Lotzestr. 16-18,
D-37083, G\"{o}ttingen, Germany}


\ead{\mailto{e.altuntac@math.uni-goettingen.de}}

\begin{abstract}

The problem of minimization of the least squares 
functional with a Fr\'{e}chet differentiable, 
lower semi-continuous, convex penalizer 
$J$ is considered to be solved. 
The penalizer maps the functions of
Banach space $\cV$ into $\R_{+},$ 
$ J : \cV \rightarrow \R_{+}.$ 
To be more precise, we also assume that some
given measured data $f^{\delta}$ is defined
on a compactly supported domain $\cZ \subset \R_{+}$ 
and in the class of Hilbert space, $f^{\delta} \in \cH = \cL^{2}(\cZ).$
Then the general Tikhonov cost functional, associated with
some given linear, compact and injective forward operator 
$\cT :\cV \rightarrow \cL^{2}(\cZ),$ is formulated as

\bea
F_{\alpha}(\varphi , f^{\delta}) : & \cV \times \cL^{2}(\cZ) & \rightarrow \R_{+} ,
\nonumber\\
& (\varphi , f^{\delta}) \longmapsto &
F_{\alpha}(\varphi , f^{\delta}) := 
\frac{1}{2}\norm{\cT\varphi - f^{\delta}}_{\cL^2(\cZ)}^2 + \alpha J(\varphi) .
\nonumber
\eea

Convergence of the regularized optimum solution
$\varphi_{\alpha(\delta)} \in \argmin_{\varphi \in \cV} F_{\alpha}(\varphi , f^{\delta})$
to the true solution $\varphi^{\dagger}$ is analysed
by means of Bregman distance.

First part of this work aims to provide some general convergence 
analysis for generally strongly convex functional $J$
in the cost functional $F_{\alpha}$. In this part the key observation
is that strong convexity of the penalty term $J$ with its convexity
modulus implies norm convergence
in the Bregman metric sense. We also
study the characterization of convergence
by means of a concave, monotonically increasing index function 
$\Psi : [0 , \infty) \rightarrow [0 , \infty) $ with 
$\Psi (0) = 0.$ In the second part, this general analysis
will be interepreted for the smoothed-TV functional ,

\bea
J_{\beta}^{\mathrm{TV}}(\varphi) := \int_{\Omega} \sqrt{\Vert \nabla \varphi(x) \Vert_2^2 + \beta} dx ,
\nonumber
\eea
where $\Omega$ is a compact and convex domain. To this end, a new lower
bound for the Hessian of $J_{\beta}^{\mathrm{TV}}$ will be estimated.
The result of this work
is applicable for any strongly convex functional.

\bigskip
\textbf{Keywords.}
{convex regularization, Bregman distance, smoothed total variation.}
\end{abstract}

\bigskip


\section{Introduction}

As alternative to well established Tikhonov regularization,
\textbf{\cite{Tikhonov63, TikhonovArsenin77}},
studying convex variational regularization with some general penalty term $J$
has become important over the last decade. Introducing a new image 
denoising method named as {\em total variation}, \textbf{\cite{RudinOsherFatemi92}},
is commencement of such study.
Application and analysis of the method have been widely carried
out in the communities of inverse problems and optimization, 
\textbf{\cite{AcarVogel94, BachmayrBurger09, BardsleyLuttman09, ChambolleLions97,
ChanChen06, ChanGolubMulet99, DobsonScherzer96, 
DobsonVogel97, VogelOman96}}. Particularly, formulating
the minimization problem as variational problem and
estimating convergence rates with variational source conditions
has also become popular recently, \textbf{\cite{BurgerOsher04, Grasmair10, 
Grasmair13, GrasmairHaltmeierScherzer11, Lorenz08}}. 

Problem of finding the optimum minimizer for a general 
Tikhonov type functional is formulated below

\beq
\label{problem0}
\varphi_{\alpha(\delta)} \in 
\argmin_{\varphi \in \cV} \left\{\frac{1}{2} \norm{\cT\varphi - f^{\delta}}_{\cH}^2 + \alpha J(\varphi) \right\} .
\eeq
Here, $ J : \cV \rightarrow \R_{+},$
is the convex penalty term and it is smooth in the Fr\'{e}chet derivative sense
with the regularization parameter $\alpha > 0$ before it.

This work aims to utilize convex analysis together with 
Bregman distance as two fundamental concepts to arrive
at convergence and convergence rates in convex regularization strategy.
In particular, it will be observed that the strong convexity
provides new quantitative analysis for the Bregman distance
which also implies norm convergence. We will interprete
this observation for the smoothed-TV functional,
\textbf{\cite{ChanGolubMulet99, DobsonScherzer96}},

\bea
J_{\beta}^{\mathrm{TV}}(\varphi) := \int_{\Omega} \sqrt{\Vert \nabla \varphi(x) \Vert_2^2 + \beta} dx .
\nonumber
\eea
Eventually, it will be shown that the strong convexity of $J_{\beta}^{\mathrm{TV}}$ 
requires the solution to be in the class of the Sobolev space $\cW^{1,2}.$

We rather focus on {\em a posteriori} strategy for the choice
of regularization parameter $\alpha = \alpha(\delta , f^{\delta}).$
and this does not require any {\em a priori} knowledge about 
the true solution. We always work with the given perturbed data
$f^{\delta}$ and introduce the rates according 
to the perturbation amount $\delta.$ Under this {\em a posteriori} strategy
and the assumed deterministic noise model, 
$f^{\delta} \in \cB_{\delta}(f^{\dagger}),$ 
in the measurement space, the following rates will be able 
to be quantified;

\begin{enumerate}
\item $\cT\varphi_{\alpha(\delta , f^{\delta})} \in \cB_{o(\delta)}(\cT\varphi^{\dagger});$
the discrepancy between 
$\cT\varphi_{\alpha(\delta , f^{\delta})}$ and $\cT\varphi^{\dagger}$
by the rate of $o(\delta).$

\item $D_{J}(\varphi_{\alpha(\delta , f^{\delta})} , \varphi^{\delta}) \leq O(\Psi(\delta));$
upper bound for the Bregman distance $D_J,$ which will immediately imply
the desired norm convergence 
$\Vert\varphi_{\alpha(\delta , f^{\delta})} - \varphi^{\dagger}\Vert_{\cV}.$

\item $\varphi_{\alpha(\delta , f^{\delta})} \in \cB_{o(\Psi(\delta))}(\varphi^{\dagger});$
convergence of the regularized solution 
$\varphi_{\alpha(\delta , f^{\delta})}$ 
to the true solution $\varphi^{\dagger}$ by the rate of
the index function $O(\Psi(\delta)).$
\end{enumerate}

\section{Notations and prerequisite knowledge}
\label{notations}

\subsection{Functional analysis notations}
\label{notations_functional_analysis}

Let $\cC(\Omega)$ be the space of continuous functions on
a compact domain $\Omega$ with its Lipschitz 
boundary $\partial\Omega.$ Then, the function space $\cC^{k}(\Omega)$ 
is defined by

\begin{displaymath}
\cC^{k}(\Omega) := \{ \varphi \in \cC(\Omega) : D^{\sigma}(\varphi) \in \cL^{p}(\Omega) \mbox{, } \forall \sigma \in \N \mbox{ with } \vert \sigma \vert \leq k\} .
\end{displaymath}
We will also need to work with Sobolev spaces. 
We define Sobolev space for $p \geq 1$ by,
\bea
\label{sobolev_sp}
\cW^{k,p}(\Omega) := \{\varphi \in \cL^{p}(\Omega) : D^{\sigma}(\varphi) \in \cL^{p}(\Omega) \mbox{, } \forall \sigma \in \N \mbox{ with } \vert \sigma \vert \leq k \} .
\nonumber
\eea
We also denote another Sobolev function space with
zero boundary value by
\bea
\label{sobolev_sp_0}
\cW_{0}^{k,p}(\Omega) := \{ \varphi \in \cC(\Omega) \vert D^{\sigma}(\varphi) \in \cL^{p}(\Omega) \mbox{ } 
\forall \sigma \in \N \mbox{ with } \vert \sigma \vert \leq k, \mbox{ and } \varphi(x) = 0 \mbox{ for } x \in \partial\Omega \} .
\nonumber
\eea
It is also worthwhile to recall the density argument, 
\textbf{\cite[Subsection 5.2.2]{Evans98}}, in $\cW^{k,p}(\Omega),$

\begin{displaymath}
\overline{C_{c}^{\infty}(\Omega)} = \cW_{0}^{k,p}(\Omega) .
\end{displaymath}

In this work, we focus on the total variation (TV) of a
$\cC^{1}$ class function. TV of a function
defined over the compact domain $\Omega$ is given below.

\begin{definition}[$TV(\varphi , \Omega)$]\textbf{\cite[Definition 9.64]{ScherzerGrasmair09}}
\label{TV_functional}
Over the compact domain $\Omega,$
total variation of a function $TV(\varphi , \Omega)$ 
is defined by the following variational form

\bea
\label{TV_functional2}
 TV(\varphi , \Omega) := \sup_{\Phi\in \cC_{c}^{1}(\Omega)} 
\left\{ \int_{\Omega} \varphi(x) \div \Phi(x) dx\mbox{ } : \mbox{ } \norm{\Phi}_{\infty} \leq 1 \right\}
\eea
\end{definition}

Total variation type regularization targets 
the reconstruction of bounded variation (BV) 
class of functions that are defined by
\bea
\label{bv_def}
BV(\Omega) := \{ \varphi \in \cL^{1}(\Omega) : TV(\varphi , \Omega) < \infty \}
\eea
with the norm
\beq
\norm{\varphi}_{BV} := \norm{\varphi}_{\cL^1} + TV(\varphi , \Omega).
\eeq
BV function spaces are Banach spaces, \textbf{\cite{Vogel02}}.
Furthermore, if a function $\varphi$ is in the class of Sobolev space $\cW^{1,1}$ 
it is also in the space of $BV(\Omega),$ (see \textbf{\cite{AcarVogel94}} and 
\textbf{\cite[Proposition 8.13]{Vogel02}}). 
By the result in \textbf{\cite[Theorem 2.1]{AcarVogel94}}, 
it is known that one can arrive, with a proper choice of $\Phi \in \cC_{c}^{1}(\Omega),$ 
at the following formulation from (\ref{TV_functional2}),

\beq
\label{tv_integral_form}
TV(\varphi) = \int_{\Omega} \norm{\nabla\varphi(x)}_2 dx \cong
\int_{\Omega} \left( \norm{\nabla\varphi(x)}_2^2 + \beta \right)^{1/2} dx = J_{\beta}^{\mathrm{TV}}(\varphi),
\eeq
where $0 < \beta < 1$ is fixed. 
We also refer to
\textbf{\cite{ChambolleLions97, ChanGolubMulet99, DobsonScherzer96, RudinOsherFatemi92, VogelOman96}}
where (\ref{tv_integral_form}) has appeared.


\subsection{Some motivation for general regularization theory}
\label{notations_regularization_theory}

For the given linear, injective and compact
forward operator $\cT : \cV \rightarrow \cH,$
over some compact and convex domain $\Omega,$
we formulate the following smooth, convex variational minimization,

\beq
\label{problem1}
\argmin_{\varphi \in \cV} \left\{\frac{1}{2} \norm{\cT\varphi - f^{\delta}}_{\cH}^2 + \alpha J(\varphi) \right\}
\eeq
with its penalty 
$ J : \cV \rightarrow \R_{+},$ 
and the regularizatin parameter $\alpha > 0.$
Another dual minimization problem to (\ref{problem1}) is given by 

\beq
\label{constrained_problem}
J(\varphi) \rightarrow \min_{\varphi \in \cV}
\mbox{, subject to } \norm{\cT\varphi - f^{\delta}}_{\cH} \leq \delta .
\eeq
Following from the problem \ref{problem1},
in what follows, the general Tikhonov type cost functional 
$F_{\alpha} : \cV \times \cH \rightarrow \R_{+}$ 
with $2-$convex penalty term $J: \cV \rightarrow \R_{+}$ is then formulated by
\bea
\label{cost_functional}
F_{\alpha}(\varphi , f^{\delta}) := 
\frac{1}{2}\norm{\cT\varphi - f^{\delta}}_{\cH}^2 + \alpha J(\varphi) .
\eea

In the Hilbert scales, it is known that the solution of 
the penalized minimizatin problem (\ref{problem1}) equals to the 
solution of the constrained minimization problem 
(\ref{constrained_problem}), \textbf{\cite[Subsection 3.1]{BurgerOsher04}}.
The regularized solution $\varphi_{\alpha(\delta)}$ 
of the problem (\ref{problem1}) satisfies the following 
first order optimality conditions,

\bea
\label{optimality_1}
& 0 & = \nabla F_{\alpha}(\varphi_{\alpha(\delta)})
\nonumber\\
& 0 & = \cT^{\ast}(\cT\varphi_{\alpha(\delta)} - f^{\delta}) + \alpha(\delta) \nabla J(\varphi_{\alpha(\delta)})
\nonumber\\
& \cT^{\ast}(f^{\delta} - \cT\varphi_{\alpha(\delta)}) & = \alpha(\delta) \nabla  J(\varphi_{\alpha(\delta)}) .
\eea

The choice of regularization parameter $\alpha(\delta , f^{\delta})$ 
in this work does not require any {\em a priori} knowledge about 
the true solution. We always work with perturbed data
$f^{\delta}$ and introduce the rates according 
to the perturbation amount $\delta.$ Throughout stability
analysis here, we consider the classical deterministic noise
model

\begin{displaymath}
f^{\delta} \in \cB_{\delta}(f^{\dagger}), \textit{\mbox{ i.e., }} 
\norm{f^{\dagger} - f^{\delta}} \leq \delta .
\end{displaymath}


\subsection{Bregman distance as a vital tool for the norm convergence}
\label{bregman_divergence_def}

\begin{definition}\textbf{[Bregman distance]}\textbf{\cite{Bregman67}}
Let $\cP : \cV \rightarrow \R \cup \{ \infty \}$ be a convex functional and
smooth in the Fr\'{e}chet derivative sense.
Then, for $u \in \cV,$ Bregman distance associated with the functional 
$\cP$ is defined by

\bea
\label{bregman_divergence_intro}
D_{\cP}(u , u^{\ast}) = \cP(u) - \cP(u^{\ast}) - \langle \nabla \cP(u^{\ast}) , u - u^{\ast} \rangle .
\eea
\end{definition}

Following formulation emphasizes the functionality of the Bregman
distance in proving the norm convergence of the minimizer
of the convex minimization problem to the true solution.

\begin{definition}\textbf{[Total convexity]}\textbf{\cite[Definition 1]{Bredies09}}
\label{total_convexity}

Let $\cP : \cV \rightarrow \R \cup \{ \infty \}$ be a Fr\'{e}chet differentiable 
convex functional. Then $\cP$ is called totally convex in $u^{\ast} \in \cV,$ 
if, 

\bea
\cP(u) - \cP(u^{\ast}) - \langle \nabla \cP(u^{\ast}) , u - u^{\ast} \rangle \rightarrow 0
\Rightarrow \norm{u - u^{\ast}}_{\cV} \rightarrow 0 .
\nonumber
\eea
It is said that $\cP$ is {\em q-convex} in $u^{\ast} \in \cV$
with a $q \in [2, \infty ),$ if for all $M > 0$ there exists a $c^{\ast} > 0$ 
such that for all $\norm{u - u^{\ast}}_{\cV} \leq M$ we have

\beq
\label{q_convexity}
\cP(u) - \cP(u^{\ast}) - \langle \nabla \cP(u^{\ast}), u - u^{\ast} \rangle \geq c^{\ast} \norm{u - u^{\ast}}_{\cV}^q .
\eeq
\end{definition}

Throughout our norm convergence estimations, we refer to this
definition for the case of $2-$convexity. We will also study
different formulations of the Bregman distance. Common usage
of the Bregman distance is to associate it with the penalty
term $J$ appears in the problem (\ref{problem1}). Here, we also
make use of different examples of the Bregman distance.

\begin{remark}\textbf{[Examples of the Bregman distance]}
\label{various_bregman}
Let $\varphi_{\alpha(\delta)} , \varphi^{\dagger} \in \cV$ 
be the regularized and the true solutions of the problem
(\ref{problem1}) respectively. Then we give the following
examples of the Bregman distance;

\begin{itemize}
\item Bregman distance associated with the cost functional $F_{\alpha}:$
\beq
\label{bregman_def_cost_func}
D_{F_{\alpha}}(\varphi_{\alpha(\delta)} , \varphi^{\dagger})
= F_{\alpha}(\varphi_{\alpha(\delta)} , f^{\delta}) - F_{\alpha}(\varphi^{\dagger} , f^{\delta}) 
- \langle \nabla F_{\alpha}(\varphi^{\dagger} , f^{\delta}) , \varphi_{\alpha(\delta)} - \varphi^{\dagger} \rangle ,
\eeq

\item Bregman distance associated with the penalty $J:$
\beq
\label{bregman_def_regularizer}
D_{J}(\varphi_{\alpha(\delta)} , \varphi^{\dagger})
= J(\varphi_{\alpha(\delta)}) - J(\varphi^{\dagger})
- \langle \nabla J(\varphi^{\dagger}) , \varphi_{\alpha(\delta)} - \varphi^{\dagger} \rangle .
\eeq

\end{itemize}
\end{remark}

Composite form of the classical Bregman distance brings
another formulation of it named as
{\em symmetrical Bregman distance}, \textbf{\cite[Definition 2.1]{Grasmair13}}, 
and defined by

\beq
\label{symmetrical_bregman}
D_{\cP}^{\mathrm{sym}}(u , u^{\ast}) := D_{\cP}(u , u^{\ast}) + D_{\cP}(u^{\ast} , u) .
\eeq
Inherently, symmetric Bregman distance is also
useful for showing norm convergence as established 
below.

\begin{propn}\textbf{\cite[as appears in the proof of Theorem 4.4]{Grasmair13}}
Let $\cP : \cV \rightarrow \R_{+} \cup \{ \infty \}$ be a smooth and q-convex functional. 
Then there exists positive constant $c^{\ast} > 0$ 
such that for all $\norm{u - u^{\ast}}_{\cV} \leq M$ we have

\bea
\label{sym_bregman_to_norm_conv}
D_{\cP}^{\mbox{sym}}(u , u^{\ast}) &=& \langle \nabla \cP(u^{\ast}) - \nabla \cP(u) , u^{\ast} - u \rangle
\nonumber\\
& \geq & c^{\ast} \norm{u - u^{\ast}}_{\cV}^2 .
\eea
\end{propn}

\begin{proof}
Proof is a straightforward result of the estimation in (\ref{q_convexity})
and the symmetrical Bregman distance definition given by (\ref{symmetrical_bregman}).
\end{proof}

In Definition \ref{total_convexity} by the estimation in (\ref{q_convexity}),
it has been stated that the norm convergence is guarenteed
in the presence of some positive real valued constant to bound
the Bregman distance, given by (\ref{bregman_divergence_intro}),
from below. It is possible to derive an alternative estimation
to (\ref{q_convexity}), or to well known
Xu-Roach inequalities in \textbf{\cite{XuRoach91}}, 
in the case of $q = 2,$ by making further assumption about the functional $\cP$
which is strong convexity with modulus $c,$
\textbf{\cite[Definition 10.5]{BauschkeCombettes11}}.
Below, we formulate the first result of this work which is the base
of our $\cL^2$ norm estimations in the analysis. 
We introduce another notation before giving our
formulation. From some reflexive Banach space $\cV$ to $\R,$
let $A,B : \cV \rightarrow \R$ and $A,B \in \cL(\R).$
Then $A \succ B $ means that $ \langle h, (A - B)h \rangle \geq 0$
for all $h \in \cV.$

\begin{propn}
\label{proposition_q-convexity}
Over the compact and convex domain $\Omega,$
let $\cP : \cL^2(\Omega) \subset \cV \rightarrow \R \cup \{ \infty \}$ 
be some strongly convex and 
twice continuously differentiable functional. 
Then the Bregman distance $D_{\cP}$ can be bounded
below by

\bea
\label{lower_bound_for_bregman}
D_{\cP}(u , v) \geq c \norm{u - v}_{\cL^2(\Omega)}^2 ,
\eea
where the modulus of convexity $c > 0$ satisfies 
$\frac{1}{2} \cP^{\prime\prime} \succ cI.$

\end{propn}

\begin{proof}

Let us begin with considering the Taylor expansion of $\cP,$
\beq
\cP(u) = \cP(v) + \langle \cP^{\prime}(v) , u - v \rangle + 
\frac{1}{2} \langle \cP^{\prime\prime}(v)(u - v) , u - v \rangle + o(\norm{u - v}_{\cL^2(\Omega)}^2) .
\eeq
where $o(\norm{u - v}_{\cL^(\Omega)}^2) := R^{2}(u - v)$ is the remainder 
given in the integral form by
\begin{displaymath}
R^{2}(u - v) = \frac{1}{6}\int_{0}^{1} \cP^{\prime\prime\prime}(v + t(u-v)) \cdot \left((1-t)(u - v)\right)^2(u-v) dt .
\end{displaymath}
Then the Bregman distance reads

\bea
D_{\cP}(u , v) & = & 
\cP(u) - \cP(v) - \langle \cP^{\prime}(v) , u - v \rangle
\nonumber\\
& = & \langle \cP^{\prime}(v) , u - v \rangle + 
\frac{1}{2} \langle \cP^{\prime\prime}(v)(u - v) , u - v \rangle + o(\norm{u - v}_{\cL^2(\Omega)}^2) 
- \langle \cP^{\prime}(v) , u - v \rangle
\nonumber\\
& = & \frac{1}{2} \langle \cP^{\prime\prime}(v)(u - v) , 
u - v \rangle + o(\norm{u - v}_{\cL^2(\Omega)}^2) .
\nonumber
\eea
Since $\cP$ is strictly convex and $o(\norm{u - v}_{\cL^2(\Omega)}^2) > 0,$ 
due to strong convexity, one eventually obtains that

\bea
D_{\cP}(u , v) \geq c \norm{u - v}_{\cL^2(\Omega)}^2 ,
\eea
where $c$ is the modulus of convexity.

\end{proof}


\subsection[Choice of regularization parameter]
{Choice of regularization parameter with Morozov's discrepancy principle}
\label{choice_of_regpar}

We are also concerned with asymptotic properties of the
regularization parameter $\alpha$ for the Tikhonov-regularized solution
obtained by Morozov's discrepancy principle.
Morozov's discrepancy principle (MDP) serves as an \textit{a posteriori}
parameter choice rule for the Tikhonov type cost functionals  
(\ref{cost_functional}) and has certain impact 
on the convergence of the regularized solution for the problem in 
(\ref{problem1}) with some general convex penalty term $J.$
As has been introduced in 
\textbf{\cite[Theorem 3.10]{AnzengruberRamlau10}} and 
\textbf{\cite{AnzengruberRamlau11}},
we will make use of the following set notations
in the theorem formulations that are necessary
to prove the norm convergence of the solution
$\varphi_{\alpha(\delta , f^{\delta})}$ to the true solution
$\varphi^{\dagger}$ for the problem (\ref{problem1}).

\bea
\label{MDP1}
\overline{S} & := & \left\{ \alpha : \norm{\cT\varphi_{\alpha(\delta)} - f^{\delta}}_{\cL^2(\cZ)} \leq \overline{\tau}\delta \mbox{ for some } 
\varphi_{\alpha(\delta)} \in \argmin_{\varphi \in \cV} \{ F_{\alpha}(\varphi , f^{\delta})\} \right\} ,
\\
\label{MDP2}
\underline{S} & := & \left\{ \alpha : \underline{\tau}\delta \leq \norm{\cT\varphi_{\alpha(\delta)} - f^{\delta}}_{\cL^2(\cZ)} \mbox{ for some } 
\varphi_{\alpha(\delta)} \in \argmin_{\varphi \in \cV} \{ F_{\alpha}(\varphi , f^{\delta})\} \right\} ,
\eea
where $1 \leq \underline{\tau} \leq \overline{\tau}$ are fixed.
Analogously, as well known from
\textbf{\cite[Eq. (4.57) and (4.58)]{Engl96}},
\textbf{\cite[Definition 2.3]{Kirsch11}},
in order to obtain tight rates of convergence of 
$\norm{\varphi_{\alpha(\delta)} - \varphi^{\dagger}}$
we are interested in sucha a regularization parameter $\alpha(\delta , f^{\delta}),$ 
with some fixed $1 \leq \underline{\tau} \leq \overline{\tau},$ that

\beq
\label{discrepancy_pr_definition}
\alpha(\delta , f^{\delta}) \in \{ \alpha > 0 \mbox{ }\vert \mbox{ }
\underline{\tau}\delta \leq \norm{\cT\varphi_{\alpha(\delta , f^{\delta})} - f^{\delta}}_{\cL^2(\cZ)} \leq \overline{\tau}\delta , 
\mbox{ for all given } (\delta , f^{\delta}) \} .
\eeq


\section{Variational Convergence Analysis}
\label{section_norm_convergence} 

Due to sophisticated nature of the TV penalty term 
in convex/non-convex minimization problems, variational
inequalities in convergence analysis
for the minimization problems in the form of (\ref{problem1})
is useful. The title name of this section solely expresses the duty
of the variational inequalities in convergence analysis. As alternative 
to well established Tikhonov regularization, 
\textbf{\cite{Tikhonov63, TikhonovArsenin77}},
studying convex regularization strategy 
has been initiated by introducing a new image denoising method
named as {\em total variation}, \textbf{\cite{RudinOsherFatemi92}}.
Particularly, formulating the minimization problem as variational problem and
estimating convergence rates with considering source conditions in variational inequalities 
has also become popular recently, 
\textbf{\cite{BurgerOsher04, Grasmair10, Grasmair13, GrasmairHaltmeierScherzer11, Lorenz08}}
and references therein.

Recall the facts that classical deterministic noise model 
$f^{\delta} \in \cB_{\delta}(f^{\dagger})$
and the $2-$convexity of the penalty term of our minimization problem (\ref{problem1})
are taken into account throughout our analysis. 
Under some {\em a posteriori} strategy 
together with the aforementioned assumptions, 
we will quantify the following rates;

\begin{enumerate}

\item $\cT\varphi_{\alpha(\delta , f^{\delta})} \in \cB_{o(\delta)}(\cT\varphi^{\dagger});$
the discrepancy between 
$\cT\varphi_{\alpha(\delta , f^{\delta})}$ and $\cT\varphi^{\dagger}$
by the rate of $o(\delta).$

\item $D_{J}(\varphi_{\alpha(\delta , f^{\delta})} , \varphi^{\delta}) \leq O(\Psi(\delta));$
upper bound for the Bregman distance $D_J,$ which will immediately imply
the desired norm convergence 
$\Vert\varphi_{\alpha(\delta , f^{\delta})} - \varphi^{\dagger}\Vert_{\cV}.$

\item $\varphi_{\alpha(\delta , f^{\delta})} \in \cB_{o(\Psi(\delta))}(\varphi^{\dagger});$
convergence of the regularized solution 
$\varphi_{\alpha(\delta , f^{\delta})}$ 
to the true solution $\varphi^{\dagger}$ by the rate of
the index function $O(\Psi(\delta)).$

\end{enumerate}

\subsection[Choice of the regularization parameter]
{Choice of the regularization parameter with Morozov's discrepancy principle}
\label{choice_of_regpar}

We are also concerned with asymptotic properties of the
regularization parameter $\alpha$ for the Tikhonov-regularized solution
obtained by Morozov's discrepancy principle.
Morozov's discrepancy principle (MDP) serves as an \textit{a posteriori}
parameter choice rule for the Tikhonov type cost functionals  
(\ref{cost_functional}) and has certain impact 
on the convergence of the regularized solution for the problem in 
(\ref{problem1}) with some general convex penalty term $J.$
As has been introduced in 
\textbf{\cite[Theorem 3.10]{AnzengruberRamlau10}} and 
\textbf{\cite{AnzengruberRamlau11}},
we will make use of the following set notations
in the theorem formulations that are necessary
to prove the norm convergence of the solution
$\varphi_{\alpha(\delta , f^{\delta})}$ to the true solution
$\varphi^{\dagger}$ for the problem (\ref{problem1}).

\bea
\label{MDP1}
\overline{S} & := & \left\{ \alpha : \norm{\cT\varphi_{\alpha(\delta)} - f^{\delta}}_{\cL^2(\cZ)} \leq \overline{\tau}\delta \mbox{ for some } 
\varphi_{\alpha(\delta)} \in \argmin_{\varphi \in \cV} \{ F_{\alpha}(\varphi , f^{\delta})\} \right\} ,
\\
\label{MDP2}
\underline{S} & := & \left\{ \alpha : \underline{\tau}\delta \leq \norm{\cT\varphi_{\alpha(\delta)} - f^{\delta}}_{\cL^2(\cZ)} \mbox{ for some } 
\varphi_{\alpha(\delta)} \in \argmin_{\varphi \in \cV} \{ F_{\alpha}(\varphi , f^{\delta})\} \right\} ,
\eea
where $1 \leq \underline{\tau} \leq \overline{\tau}$ are fixed.
Analogously, as well known from
\textbf{\cite[Eq. (4.57) and (4.58)]{Engl96}},
\textbf{\cite[Definition 2.3]{Kirsch11}},
in order to obtain tight rates of convergence of 
$\norm{\varphi_{\alpha(\delta)} - \varphi^{\dagger}}$
we are interested in such a regularization parameter $\alpha(\delta , f^{\delta}),$ 
with some fixed $1 \leq \underline{\tau} \leq \overline{\tau},$ that

\beq
\label{discrepancy_pr_definition}
\alpha(\delta , f^{\delta}) \in \{ \alpha > 0 \mbox{ }\vert \mbox{ }
\underline{\tau}\delta \leq \norm{\cT\varphi_{\alpha(\delta , f^{\delta})} - f^{\delta}}_{\cL^2(\cZ)} \leq \overline{\tau}\delta \} 
\mbox{ for all given } (\delta , f^{\delta}) .
\eeq


\subsection{Variational inequalities for norm convergence}
\label{variational_ineq}

Convergence rates results for some general operator $\cT$
can be obtained by formulating {\em variational inequality}
which uses the concept of index functions.  A function
$\Psi : [0 , \infty) \rightarrow [0 , \infty)$
is called {\em index function} if it is continuously defined, 
monotonically increasing and $\Psi(0) = 0.$

\begin{assump}
\label{assump_variational_ineq}
\textbf{[Variational Inequality]}\textbf{\cite[Eq. 1]{Grasmair10}, \cite[Eq 1.5]{HofmannMathe12}, \cite[Eq 2]{HohageWeidling15}}
There exists some constant $\tilde{\gamma} \in (0 , 1]$ 
and an index function $\Psi,$ for all 
$\varphi \in  \cD(\cT) , $ such that
 
\bea
\label{variational_ineq}
\tilde{\gamma}\norm{\varphi - \varphi^{\dagger}}_{\cL^{2}(\Omega)}^{2} \leq J(\varphi) - J(\varphi^{\dagger})
+ \Psi\left( \norm{\cT\varphi - \cT\varphi^{\dagger}}_{\cL^{2}(\cZ)} \right) .
\eea 

\end{assump}

\begin{lemma}
\label{grad_penalty_upper_bound}
For the cost functional defined by

\begin{displaymath}
F_{\alpha}(\varphi , f^{\delta}) := 
\frac{1}{2}\norm{\cT\varphi - f^{\delta}}_{\cL^2(\cZ)}^{2} + \alpha J(\varphi),
\end{displaymath}
with some Fr\'{e}chet differentiable and convex penalty term $J : \cV \rightarrow \R,$
that is defined on a Hilbert space $\cV,$ $J : \cV \rightarrow \R,$
let $\varphi_{\alpha} \in \argmin_{\varphi \in \cV} \{F_{\alpha}(\varphi , f^{\delta})\}.$
Then for all $\varphi \in \cD(\cT) \subset \cV$ and any regularization parameter $\alpha > 0,$

\bea
\label{upper_bound}
\alpha \langle \nabla J(\varphi) , \varphi_{\alpha} - \varphi \rangle
\leq \langle \cT^{\ast}(\cT\varphi - f^{\delta}) ,  \varphi - \varphi_{\alpha} \rangle .
\eea
\end{lemma}

\begin{proof}

Since $\varphi_{\alpha}$ is the minimum of the cost functional $F_{\alpha}$
then, it is hold that  
$F_{\alpha}(\varphi_{\alpha} , f^{\delta}) \leq F_{\alpha}(\varphi , f^{\delta})$ 
for all $\varphi \in \cD(\cT) \subset \cV$ and $\alpha > 0.$
Now, recall the Bregman distance formulation associated with the cost functional 
in (\ref{bregman_def_cost_func}).

\bea
0 \leq D_{F_{\alpha}}(\varphi_{\alpha} , \varphi) & = &
F_{\alpha}(\varphi_{\alpha}) - F_{\alpha}(\varphi) - \langle \nabla F_{\alpha}(\varphi) , \varphi_{\alpha} - \varphi \rangle
\nonumber\\
& \leq & - \langle \nabla F_{\alpha}(\varphi) , \varphi_{\alpha} - \varphi \rangle 
\nonumber\\
& = & \langle \nabla F_{\alpha}(\varphi) , \varphi - \varphi_{\alpha} \rangle
\eea
We, by the definition of the cost functional $F_{\alpha}$ in (\ref{cost_functional}),
have that

\bea
0 \leq \langle \cT^{\ast}(\cT\varphi - f^{\delta}) + \alpha \nabla J(\varphi), 
\varphi - \varphi_{\alpha} \rangle  ,
\eea
which yields the assertion.
\end{proof}

It is also an immediate consequence of MDP, see
\textbf{\cite[Remark 2.7]{AnzengruberRamlau11}},
that

\bea
\label{consequence_MDP}
\norm{\cT\varphi_{\alpha(\delta)} - \cT\varphi^{\dagger}}_{\cL^2(\cZ)} \leq (\overline{\tau} + 1)\delta .
\eea
We use this observation to formulate the following theorem.
The first assertion below is an expected result
for minimization problems given by (\ref{problem1}),
see \textit{e.g.} \textbf{\cite[Lemma 1]{HofmannMathe12}}.

\begin{theorem}
\label{Bregman_regularizer_upper_bound2}
Under the same assumption in Lemma \ref{grad_penalty_upper_bound}
together with $\norm{\cT\varphi^{\dagger} - f^{\delta}}_{\cL^{2}(\cZ)} \leq \delta,$ 
then we, for any $\alpha > 0,$ have that

\bea
\label{pointwise_conv_penalty}
J(\varphi_{\alpha}) - J(\varphi^{\dagger}) & \leq & \frac{\delta^2}{2\alpha} .
\eea
Moreover, for $\alpha(\delta , f^{\delta}) \in \overline{S},$ the Bregman distance
$D_J$ is bounded above by

\bea
\label{bregman_upper_bound}
D_{J}(\varphi_{\alpha(\delta , f^{\delta})} , \varphi^{\dagger}) & \leq & 
\frac{\delta^2}{\alpha(\delta , f^{\delta})}\left( \frac{3}{2} + \overline{\tau}\right).
\eea
\end{theorem}

\begin{proof}
Since $\varphi_{\alpha},$ for any $\alpha > 0,$ 
is the minimizer of the cost functional $F_{\alpha},$ then

\bea
F_{\alpha}(\varphi_{\alpha} , f^{\delta}) & = & 
\frac{1}{2}\norm{\cT \varphi_{\alpha} - f^{\delta}}_{\cL^{2}(\cZ)}^2 
+ \alpha J(\varphi_{\alpha})
\nonumber\\
& \leq & \frac{1}{2}\norm{\cT \varphi^{\dagger} - f^{\delta}}_{\cL^{2}(\cZ)}^2 + 
\alpha J(\varphi^{\dagger}) = F_{\alpha}(\varphi^{\dagger} , f^{\delta}),
\nonumber
\eea
which is in other words,

\beq
\alpha (J(\varphi_{\alpha}) - J(\varphi^{\dagger}))
\leq \frac{1}{2} \norm{\cT \varphi^{\dagger} - f^{\delta}}_{\cL^{2}(\cZ)}^2 - 
\frac{1}{2} \norm{\cT \varphi_{\alpha} - f^{\delta}}_{\cL^{2}(\cZ)}^2 .
\eeq
By the assumed deterministic noise model $\norm{\cT\varphi^{\dagger} - f^{\delta}}_{\cL^{2}(\cZ)} \leq \delta$
and the fact that $\norm{\cT \varphi_{\alpha} - f^{\delta}}_{\cL^{2}(\cZ)} > 0,$
one obtains the first assertion

\bea
J(\varphi_{\alpha}) - J(\varphi^{\dagger}) \leq \frac{\delta^2}{2\alpha} .
\nonumber
\eea
Regarding second assertion, since $\alpha(\delta , f^{\delta}) \in \overline{S},$
by the definition in (\ref{MDP1}), 
$\norm{\cT\varphi_{\alpha(\delta , f^{\delta})} - f^{\delta}}_{\cL^2(\cZ)} \leq \overline{\tau}\delta.$
From the formulation of Bregman distance (\ref{bregman_def_regularizer}) and 
Lemma \ref{grad_penalty_upper_bound}, we obtain

\bea
D_{J}(\varphi_{\alpha(\delta , f^{\delta})} , \varphi^{\dagger}) & \leq & 
\left\vert J(\varphi_{\alpha(\delta , f^{\delta})}) - J(\varphi^{\dagger}) \right\vert + 
\left\vert \langle \nabla J(\varphi^{\dagger}) , \varphi_{\alpha(\delta , f^{\delta})} - \varphi^{\dagger} \rangle \right\vert
\nonumber\\
& \leq & \frac{\delta^2}{2\alpha(\delta , f^{\delta})} + \frac{\delta}{\alpha(\delta , f^{\delta})} \norm{\cT\varphi_{\alpha(\delta , f^{\delta})} - \cT\varphi^{\dagger}}_{\cL^{2}(\cZ)} .
\nonumber
\eea
Hence, the observation in (\ref{consequence_MDP})
yields the second assertion.

\end{proof}

Obtaining tight rates of convergence with an {\em a posteriori}
strategy for the choice of regularization parameter 
$\alpha = \alpha(\delta , f^{\delta})$ is the aim of this chapter.
Henceforth, we will show the impact of this strategy on
the convergence and convergence rates by 
associating it with the index function $\Psi$ that has appeared
in Assumption \ref{assump_variational_ineq}. 
In \textbf{\cite[Eq (3.2)]{HofmannMathe12}}, a reasonable 
index function has been introduced.  We, in analogous with
that function in the regarding work, introduce

\bea
\label{index_function}
\Phi(\delta , f^{\delta}) := \frac{\delta^2}{2\Psi( \delta )} .
\eea
From this index function,
it is possible to be able to formulate an improved 
counterpart of the result in \textbf{\cite[Corollary 1]{HofmannMathe12}}.
Firstly, we give a preliminary estimate result based on the variational
inequality.

\begin{lemma}\textbf{\cite[Lemma 2]{HofmannMathe12}}
\label{HomannMathe12_Lemma2}
Let, for some $\alpha,$
$\varphi_{\alpha} \in \argmin_{\varphi \in \cV} \{F_{\alpha}(\varphi , f^{\delta})\}$
satisfy Assumption (\ref{assump_variational_ineq}). Then

\bea
\norm{\cT\varphi_{\alpha} - \cT\varphi^{\dagger}}_{\cL^2({\cZ})}^2 \leq 4\delta^2 + 
4 \alpha \Psi\left( \norm{\cT\varphi_{\alpha}  - \cT\varphi^{\dagger}}_{\cL^{2}(\cZ)} \right) ,
\nonumber
\eea
where $\varphi^{\dagger} \in \cD(\cT)$ is the true solution for the problem (\ref{problem1}).
\end{lemma}

We are now ready to introduce our result which is comparable with 
\textbf{\cite[Corollary 1]{HofmannMathe12}}. In our formulation,
we still follow {\em a posteriori} rule of choice of the regularization
parameter $\alpha = \alpha(\delta , f^{\delta}) \in \overline{S}$ 
as has been introduced in (\ref{MDP1}).

\begin{cor}
\label{HofmannMathe12_Corollary1}
Under the same assumption in Lemma \ref{HomannMathe12_Lemma2},
if the regularization parameter $\alpha(\delta , f^{\delta}) \in \overline{S}$
is chosen as

\beq
\alpha(\delta , f^{\delta}) := \Phi(\delta , f^{\delta}),
\eeq
then we have

\bea
\norm{\cT\varphi_{\alpha(\delta , f^{\delta})} - \cT\varphi^{\dagger}}_{\cL^2(\cZ)} \leq \delta\sqrt{6 + 2\overline{\tau}} ,
\eea
where fixed $1 \leq \overline{\tau}$ satisfies
$\Vert\cT\varphi_{\alpha(\delta , f^{\delta})} - f^{\delta}\Vert_{\cL^2(\cZ)} \leq \overline{\tau}\delta.$
\end{cor}

\begin{proof}

By the defined index function in (\ref{index_function}) and
the result in Lemma \ref{HomannMathe12_Lemma2},
we immediately obtain,

\bea
\norm{\cT\varphi_{\alpha(\delta , f^{\delta})} - \cT\varphi^{\dagger}}_{\cL^2(\cZ)}^{2}
& \leq & 4\delta^2 + 4\alpha(\delta , f^{\delta}) \Psi\left(\norm{\cT\varphi_{\alpha(\delta , f^{\delta})} - \cT\varphi^{\dagger}}_{\cL^2(\cZ)}\right)
\nonumber\\
& = ^{\footnotemark} & 4\delta^2 + 2\frac{\delta^2}{\Psi( \delta )} \Psi\left(\frac{\delta}{\delta}\norm{\cT\varphi_{\alpha(\delta , f^{\delta})} - \cT\varphi^{\dagger}}_{\cL^2(\cZ)}\right)
\nonumber\\
& = ^{\footnotemark} & 4\delta^2 + 2\delta \norm{\cT\varphi_{\alpha(\delta , f^{\delta})} - \cT\varphi^{\dagger}}_{\cL^2(\cZ)}
\nonumber\\
& \leq ^{\footnotemark} & 4\delta^2 + 2\delta^2 (\overline{\tau} + 1) = \delta^2(6 + 2\overline{\tau}) .
\nonumber
\eea
\footnotetext[1]{By the given index function $\Phi$ in (\ref{index_function}) and since $\alpha(\delta , f^{\delta}) := \Phi(\delta , f^{\delta}),$ 
the equation follows.}
\footnotetext[2]{See \textbf{\cite[Proposition 1]{HofmannMathe12}}.}
\footnotetext[3]{Since $\alpha(\delta , f^{\delta}) \in \overline{S},$ then 
$\Vert\cT\varphi_{\alpha(\delta , f^{\delta})} - f^{\delta}\Vert_{\cL^2(\cZ)} \leq \overline{\tau}\delta.$
This, by the triangle inequality, implies that 
$\Vert\cT\varphi_{\alpha(\delta , f^{\delta})} - \cT\varphi^{\dagger}\Vert_{\cL^2(\cZ)} \leq (\overline{\tau} + 1)\delta.$}

\end{proof}


With the introduced index function in (\ref{index_function}),
it is essential to be able to find lower bound for 
the regularization parameter $\alpha.$

\begin{cor}
\label{cor_regpar_lower_bound}
Suppose that, for a chosen regularization parameter $\alpha(\delta , f^{\delta}) \in \underline{S}$
that is defined in (\ref{MDP2}),
the regularized solution $\varphi_{\alpha(\delta , f^{\delta})}$
to the problem (\ref{problem1}) 
satisfies the variational inequality in Assumption \ref{assump_variational_ineq}.
Then the regularization parameter can be bounded below as such,

\bea
\frac{1}{2}(\underline{\tau} - 1)^2 \frac{\underline{\tau}^2 - 1}{\underline{\tau}^2 + 1} \Phi(\delta , f^{\delta}) \leq \alpha(\delta , f^{\delta}) .
\eea

\end{cor}

\begin{proof}

Since $\norm{\cT\varphi_{\alpha(\delta , f^{\delta})} - f^{\delta}}_{\cL^2({\cZ})} \geq \underline{\tau}\delta$
and the regularized solution $\varphi_{\alpha(\delta , f^{\delta})}$ satisfies
the assertion in Assumption \ref{assump_variational_ineq}, we immediately obtain,

\bea
\frac{\underline{\tau}^2 \delta^2}{2} \leq \frac{\norm{\cT\varphi_{\alpha(\delta , f^{\delta})} - f^{\delta}}_{\cL^2({\cZ})}^2}{2}
& \leq & \frac{\delta^2}{2} + \alpha \left( J(\varphi^{\dagger}) - J(\varphi_{\alpha(\delta)})  \right)
\nonumber\\
& \leq & \frac{\delta^2}{2} + \alpha \Psi\left(\norm{\cT\varphi_{\alpha(\delta , f^{\delta})} - \cT\varphi^{\dagger}}_{\cL^2({\cZ})}\right) ,
\nonumber
\eea
and this follows up

\bea
\delta^2 \leq \frac{2\alpha}{\underline{\tau}^2 - 1}\Psi\left(\norm{\cT\varphi_{\alpha(\delta , f^{\delta})} - \cT\varphi^{\dagger}}_{\cL^2({\cZ})}\right) . 
\eea
We plug this into the bound in Lemma \ref{HomannMathe12_Lemma2}
with the abbreviation $p_{\alpha} := \norm{\cT\varphi_{\alpha(\delta , f^{\delta})} - \cT\varphi^{\dagger}}_{\cL^{2}(\cZ)}$

\bea
\label{regpar_lower_bound}
p_{\alpha}^2 \leq 4\delta^2 + 4 \alpha \Psi(p_{\alpha}) 
& \leq & \frac{8\alpha}{\underline{\tau}^2 - 1}\Psi(p_{\alpha})  + 4\alpha \Psi(p_{\alpha}) 
\nonumber\\
& = & 4\alpha \Psi(p_{\alpha}) \frac{\underline{\tau}^2 + 1}{\underline{\tau}^2 - 1} .
\eea
Note that

\bea
\nonumber
\underline{\tau}\delta \leq \norm{\cT\varphi_{\alpha(\delta , f^{\delta})} - f^{\delta}}_{\cL^{2}(\cZ)} \leq 
p_{\alpha} + \delta
\eea
which implies

\bea
(\underline{\tau} - 1)\delta \leq p_{\alpha} .
\eea
Hence, from (\ref{regpar_lower_bound}),

\bea
\frac{1}{2}(\underline{\tau} - 1)^2 \frac{\underline{\tau}^2 - 1}{\underline{\tau}^2 + 1} \Phi(\delta , f^{\delta}) \leq \alpha .
\eea

\end{proof}


\begin{theorem}
\label{thrm_Bregman_upper_bound}
Suppose that the regularized solution $\varphi_{\alpha(\delta , f^{\delta})}$
to the problem (\ref{problem1}) obeys 
Assumption \ref{assump_variational_ineq},
for some regularization parameter $\alpha(\delta , f^{\delta})$
satisfying

\bea
\underline{\tau}\delta \leq \norm{\cT\varphi_{\alpha(\delta , f^{\delta})} - f^{\delta}}_{\cL^2(\cZ)} \leq \overline{\tau}\delta ,
\nonumber
\eea
where $1 \leq \underline{\tau} \leq \overline{\tau}$ are fixed
and with the lower bound in Corollary \ref{cor_regpar_lower_bound}. 
Then, by the second assertion (\ref{bregman_upper_bound})
in Theorem \ref{Bregman_regularizer_upper_bound2}, 
the Bregman distance $D_J$ can be bounded by

\bea
D_{J}(\varphi_{\alpha(\delta , f^{\delta})} , \varphi^{\dagger}) \leq 
O\left( \Psi\left( \delta \right) \right) .
\eea

\end{theorem}

\begin{proof}
Corrollary \ref{cor_regpar_lower_bound} and
the index function defined by (\ref{index_function})
provide the result

\bea
D_{J}(\varphi_{\alpha(\delta , f^{\delta})} , \varphi^{\dagger}) & \leq &
\frac{\delta^2}{\frac{1}{2}(\underline{\tau} - 1)^2 \frac{\underline{\tau}^2 - 1}{\underline{\tau}^2 + 1} \Phi(\delta , f^{\delta})}
\left( \frac{3}{2} + \overline{\tau} \right)
\nonumber\\
& = & \frac{4 \Psi( \delta ) (\underline{\tau}^2 + 1)}
{(\underline{\tau} - 1)^3 (\underline{\tau} + 1)} \left( \frac{3}{2} + \overline{\tau} \right) .
\nonumber
\eea

\end{proof}

%
%
%
%


\begin{theorem}
\label{Bregman_regularizer_upper_bound}
Let $\cT : \cL^{2}(\Omega) \rightarrow \cL^{2}(\cZ)$ 
be the compact and linear operator.
Over the compact and convex domain $\Omega,$
let $\varphi_{\alpha(\delta , f^{\delta})} \in \cL^{2}(\Omega)$ 
satisfy the assumption of Lemma \ref{grad_penalty_upper_bound} and 
Assumption \ref{assump_variational_ineq}.
If the regularization parameter $\alpha(\delta , f^{\delta}) \in \overline{S}$ 
is chosen as $\alpha(\delta , f^{\delta}) := \Phi(\delta , f^{\delta})$ 
where $\Phi$ is defined by (\ref{index_function})
with some given noisy measurement 
$f^{\delta} \in \cB_{\delta}(f^{\dagger}),$ 
then one can find the following upper bound
for the symmetric Bregman distance,

\bea
D_{J}(\varphi_{\alpha(\delta , f^{\delta})} , \varphi^{\dagger}) \leq 
D_{J}^{\mathrm{sym}}(\varphi_{\alpha(\delta , f^{\delta})} , \varphi^{\dagger}) \leq 
\frac{1}{\epsilon} \left(\frac{1}{\tilde{\gamma}}  + 1 \right) \Psi(\delta) + 
\epsilon \Psi^{2}(\delta)\norm{\cT^{\ast}}^2 (\overline{\tau}^2 + 1) ,
\nonumber
\eea
where the coefficients are arbitrarily chosen as $\epsilon \in \R_{+},$
$\tilde{\gamma} \in (0,1],$ and $\overline{\tau} \geq 1.$
Furthermore, if the smooth penalty term 
$J : \cL^{2}(\Omega) \rightarrow \R$ is $2-$convex, then
this upper bound implies,

\beq
\norm{\varphi_{\alpha(\delta , f^{\delta})} - \varphi^{\dagger}}_{\cL^{2}(\Omega)}^2 \leq O(\Psi(\delta)) .
\eeq
\end{theorem}

\begin{proof} 
By the definition of $D_{J}^{\mathrm{sym}}$ in (\ref{symmetrical_bregman}),
it suffices to prove the last inequality. First, observe that,

\bea
D_{J}^{\mathrm{sym}}(\varphi_{\alpha(\delta , f^{\delta})} , \varphi^{\dagger})
& = & \langle \nabla J(\varphi_{\alpha(\delta , f^{\delta})}) - \nabla J(\varphi^{\dagger}) , 
\varphi_{\alpha(\delta , f^{\delta})} - \varphi^{\dagger} \rangle .
\nonumber\\
& \leq & \vert \langle \nabla J(\varphi_{\alpha(\delta , f^{\delta})}) , 
\varphi_{\alpha(\delta , f^{\delta})} - \varphi^{\dagger} \rangle \vert
+ \vert \langle \nabla J(\varphi^{\dagger}) , \varphi_{\alpha(\delta , f^{\delta})} - \varphi^{\dagger} \rangle \vert .
\nonumber
\eea
We will bound each inner product separately.
The regularized solution $\varphi_{\alpha(\delta , f^{\delta})},$ 
for the regularization parameter $\alpha(\delta , f^{\delta}) := \Phi(\delta , f^{\delta})$ 
where $\Phi$ is defined by (\ref{index_function}), satisfies the first order
optimality condition given in (\ref{optimality_1}) as well as the variational
inequality in Assumption \ref{assump_variational_ineq}. So,

\bea
\left\vert \langle \nabla J(\varphi_{\alpha(\delta , f^{\delta})}) , 
\varphi_{\alpha(\delta , f^{\delta})} - \varphi^{\dagger} \rangle \right\vert
& = & 
\left\vert \frac{1}{\alpha(\delta , f^{\delta})} \langle \cT^{\ast}(f^{\delta} - \cT\varphi_{\alpha(\delta , f^{\delta})}) , 
\varphi_{\alpha(\delta)} - \varphi^{\dagger} \rangle \right\vert
\nonumber\\
& \leq & \frac{1}{\alpha(\delta , f^{\delta})} \norm{\cT^{\ast}} 
\norm{\cT\varphi_{\alpha(\delta , f^{\delta})} - f^{\delta}}_{\cL^{2}(\cZ)}
\norm{\varphi_{\alpha(\delta , f^{\delta})} - \varphi^{\dagger}}_{\cL^{2}(\Omega)}
\nonumber\\
& \leq ^{\footnotemark}& 
\frac{\epsilon}{2\alpha^2(\delta , f^{\delta})} \norm{\cT^{\ast}}^2\delta^2\overline{\tau}^2 +
\frac{1}{2\epsilon} \norm{\varphi_{\alpha(\delta , f^{\delta})} - \varphi^{\dagger}}_{\cL^2(\Omega)}^2
\nonumber\\
& \leq & \frac{\epsilon}{2\alpha^2(\delta , f^{\delta})} \norm{\cT^{\ast}}^2\delta^2\overline{\tau}^2 
+ \frac{1}{2\epsilon}\left(\frac{1}{\tilde{\gamma}}\frac{\delta^2}{2\alpha(\delta , f^{\delta})} + \Psi(\delta) \right)
\nonumber
\eea
\footnotetext[1]{By Young's inequality and since $\alpha(\delta , f^{\delta}) \in \overline{S}.$}

\noindent The assertion in Lemma \ref{grad_penalty_upper_bound},
with the regularization parameter $\alpha(\delta , f^{\delta}) > 0,$
brings the following bound

\bea
\vert \langle \nabla J(\varphi^{\dagger}) , \varphi_{\alpha(\delta , f^{\delta})} - \varphi^{\dagger} \rangle \vert
& \leq & \left\vert \frac{1}{\alpha(\delta , f^{\delta})} \langle \cT^{\ast}(f^{\dagger} - f^{\delta}) , 
\varphi^{\dagger} - \varphi_{\alpha(\delta , f)} \rangle \right\vert
\nonumber\\
& \leq & \frac{\delta}{\alpha(\delta , f^{\delta})} \norm{\cT^{\ast}} \norm{\varphi_{\alpha(\delta , f^{\delta})} - \varphi^{\dagger}}_{\cL^{2}(\Omega)} .
\nonumber\\
& \leq ^{\footnotemark} & \frac{\epsilon}{2\alpha^2(\delta , f^{\delta})} \norm{\cT^{\ast}}^2\delta^2 + 
\frac{1}{2\epsilon}\norm{\varphi_{\alpha(\delta , f^{\delta})} - \varphi^{\dagger}}_{\cL^{2}(\Omega)}^2
\nonumber\\
& \leq & \frac{\epsilon}{2\alpha^2(\delta , f^{\delta})} \norm{\cT^{\ast}}^2\delta^2 + 
\frac{1}{2\epsilon}\left(\frac{1}{\tilde{\gamma}}\frac{\delta^2}{2\alpha(\delta , f^{\delta})} + \Psi(\delta) \right) .
\nonumber
\eea
\footnotetext[2]{By Young's inequality and since $\alpha(\delta , f^{\delta}) \in \overline{S}.$}

\noindent Since the regularization parameter is chosen as
$\alpha(\delta , f^{\delta}) := \Phi(\delta , f^{\delta}),$
see (\ref{index_function}), then

\bea
D_{J}^{\mathrm{sym}} (\varphi_{\alpha(\delta , f^{\delta})} , \varphi^{\dagger}) & \leq & 
\frac{1}{2\epsilon}\left(\frac{1}{\tilde{\gamma}}\frac{\delta^2}{2\alpha(\delta , f^{\delta})} + \Psi(\delta) \right) 
+ \frac{\epsilon}{2\alpha^2(\delta , f^{\delta})} \norm{\cT^{\ast}}^2\delta^2(\overline{\tau}^2 + 1) .
\nonumber
\eea
With the additional assumption on $J$ which is $2-$convexity, 
then the norm convergence of 
$\norm{\varphi_{\alpha(\delta , f^{\delta})} - \varphi^{\dagger}}_{\cL^{2}(\Omega)}$
is obtained due to (\ref{q_convexity}).
\end{proof}



\section{Convex Regularization for the Smoothed-TV}
\label{minimize_convex_integrand}

In this section, we give the specific interpretation
of the general convex regularization for the $2-$convex,
see (\ref{q_convexity}) in Definition \ref{total_convexity}, 
smoothed-TV functional. To this end, we state
the following minimization problem

\bea
\label{smoothed-tv_min0}
\varphi_{\alpha(\delta)} \in \argmin_{\varphi \in \cW^{1,2}(\Omega)} 
\left\{\frac{1}{2} \norm{\cT\varphi - f^{\delta}}_{\cL^2}^2 + 
\alpha J_{\beta}^{\mathrm{TV}}(\varphi) \right\} ,
\eea
where the smoothed-TV penalty, 
\textbf{\cite{ChanGolubMulet99, DobsonScherzer96}}, 
is defined by 

\begin{displaymath}
J_{\beta}^{\mathrm{TV}}(\varphi) := \int_{\Omega} \sqrt{\norm{\nabla\varphi(x)}_2^2 + \beta} dx .
\end{displaymath}
Existence of the solution for the problem (\ref{smoothed-tv_min0}) has been studied extensively in
\textbf{\cite{AcarVogel94, Hintermuller14, Setzer11}}. Moreover, an existence and uniquness 
theorem for the minimizer of quadratic functionals 
with different type of convex integrands has been established 
in \textbf{\cite[Theorem 9.5-2]{Ciarlet13}}. 
As has been given by the {\em Minimal Hypersurfaces} problem in \textbf{\cite{Ekeland74}}, the minimizer
of the problem (\ref{smoothed-tv_min0})
exists on the Hilbert space $\cW^{1,2}(\Omega).$ 

Unlike in the available literature \textbf{\cite{AcarVogel94, BachmayrBurger09, BardsleyLuttman09, 
ChambolleLions97, ChanChen06, ChanGolubMulet99, DobsonScherzer96, 
DobsonVogel97, VogelOman96}}, we will arrive at a new lower bound
for the Bregman distance particularly associated with
the smoothed-TV functional $J_{\beta}^{\mathrm{TV}}.$
We will achieve this by means of the strong convexity 
of the regarding functional.

\begin{theorem}\textbf{[Convexity of the smoothed-TV penalty]}\textbf{\cite[Theorem 2.4]{AcarVogel94}}
For any $\beta > 0,$ the functional 
$J_{\beta}^{\mathrm{TV}} : \cW^{1,1}(\Omega) \rightarrow \R_{+},$ that is defined by 
$J_{\beta}^{\mathrm{TV}}(\varphi) := \sqrt{\Vert \nabla\varphi(x) \Vert_2^2 + \beta} dx,$ 
is convex.
\end{theorem}

Before the Hessian of $J_{\beta}^{\mathrm{TV}},$
we first calculate the Fr\'{e}chet derivative
of it in the direction $\Phi \in \cC_{c}^{1}(\Omega),$ 

\bea
\label{derivative_smoothed_tv}
\frac{d}{d t} J_{\beta}^{\mathrm{TV}}(\varphi + t \Phi) |_{t = 0} 
& = & \int_{\Omega} \frac{(\nabla \varphi(x) + t \nabla\Phi(x))\nabla\Phi(x)}{(\Vert \nabla \varphi(x) + t \nabla\Phi(x)\Vert_2^{2}+\beta)^{1/2}} dx |_{t = 0}
\nonumber\\
& = & \int_{\Omega} \frac{\nabla\varphi(x) \nabla\Phi(x)}{(\Vert \nabla \varphi(x) \Vert_2^2 + \beta)^{1/2}} dx.
\nonumber\\
& = & \int_{\Omega} \nabla^{\ast} \left( \frac{\nabla\varphi(x) }{(\Vert \nabla \varphi(x) \Vert_2^2 + \beta)^{1/2}} \right) \Phi(x) dx ,
\eea 
where $\nabla^{\ast}$ represents the adjoint of 
the gradient operator which is $\nabla^{\ast} = - \mathrm{div} .$

%

\begin{theorem}\textbf{[Smoothed-TV functional is strongly convex]}
\label{smoothed_TV_strongly_convex}
For any $\varphi \in \cW^{1,2}(\Omega)$ defined
over the compactly supported domain $\Omega \subset \R^{3}$
and for the smoothed-TV functional 
$J_{\beta}^{\mathrm{TV}} : \cW^{1,2}(\Omega) \rightarrow \R_{+},$ 

\begin{displaymath}
J_{\beta}^{\mathrm{TV}}(\varphi) := \int_{\Omega} \sqrt{\norm{\nabla \varphi(x)}_2^2 + \beta} dx ,
\end{displaymath}
where $\beta > 0$ is fixed, the Hessian of
of $J_{\beta}^{\mathrm{TV}},$ which is
$(J_{\beta}^{\mathrm{TV}})^{\prime\prime}[\varphi](\Phi,\Phi),$
can be bounded below by some functional 
$l : \cW^{1,2}(\Omega)  \rightarrow \R_{+}$
satisfying

\begin{displaymath}
(J_{\beta}^{\mathrm{TV}})^{\prime\prime}[\varphi](\Phi,\Phi) \geq l(\varphi) \norm{\nabla\Phi}_{\cL^2(\Omega)}^2 .
\end{displaymath}

\end{theorem}

\begin{proof}

In (\ref{derivative_smoothed_tv}), we, in the direction
$\Phi \in \cC_{c}^{1}(\Omega)$, have calculated that

\bea
(J_{\beta}^{\mathrm{TV}})^{\prime}[\varphi] (\Phi) = 
\int_{\Omega} \frac{\nabla\varphi(x) \nabla\Phi(x)}{(\Vert \nabla \varphi(x) \Vert_2^2 + \beta)^{1/2}} dx .
\nonumber
\eea
Following from here, we can calculate the Hessian in the direction 
$\Psi \in \cW^{1,2}(\Omega),$

\bea
(J_{\beta}^{\mathrm{TV}})^{\prime\prime}[\varphi] (\Phi , \Phi) & = & 
\frac{d}{d s}( J_{\beta}^{\mathrm{TV}})^{\prime}[\varphi + s\Psi] (\Phi) \Bigg\vert_{s = 0}
\nonumber\\
& = &
\frac{d}{d s} \int_{\Omega}  \frac{(\nabla \varphi(x) + s\nabla \Psi(x))\nabla\Phi(x)}
{(\Vert\nabla\varphi(x) + s\nabla\Psi(x) \Vert_2^2 + \beta)^{1/2}} dx \Bigg\vert_{s = 0} ,
\nonumber
\eea
which is


\bea
\frac{d}{d s}( J_{\beta}^{\mathrm{TV}})^{\prime}[\varphi + s\Psi] (\Phi) \Bigg\vert_{s = 0}
& = & \int_{\Omega} \frac{\nabla\Psi(x)\nabla\Phi(x)(\Vert\nabla\varphi(x)\Vert_2^2 + \beta) - 
\Vert \nabla\varphi(x) \nabla\Phi(x)\Vert^2}{(\Vert \nabla\varphi(x) \Vert_2^2 + \beta)^{3/2}} dx .
\nonumber
\eea
By the choice of $\Psi = \Phi$ and since 
$\Vert \nabla\varphi\nabla\Phi\Vert_2^2 \leq \Vert \nabla\varphi \Vert_2^2 \Vert\nabla\Phi\Vert_2^2 ,$ 

\bea
(J_{\beta}^{\mathrm{TV}})^{\prime\prime}[\varphi](\Phi,\Phi) & \geq &
\int_{\Omega} \frac{\Vert \nabla\Phi(x) \Vert_2^2(\Vert\nabla\varphi(x)\Vert_2^2 + \beta) - 
\Vert \nabla\varphi \Vert_2^2 \Vert\nabla\Phi\Vert_2^2}{(\Vert \nabla\varphi(x) \Vert_2^2 + \beta)^{3/2}} dx 
\nonumber\\
& = & \int_{\Omega} \frac{\beta \Vert \nabla\Phi(x) \Vert_2^2}{(\Vert \nabla\varphi(x) \Vert_2^2 + \beta)^{2}} dx .
\eea
Now, let us abbreviate 

\begin{displaymath}
m_{\beta}(\varphi) := 
\frac{\beta}{(\Vert \nabla\varphi \Vert_2^2 + \beta)^{2}} .
\end{displaymath}
Hence, we have

\begin{displaymath}
(J_{\beta}^{\mathrm{TV}})^{\prime\prime}[\varphi](\Phi,\Phi) \geq 
\inf_{x \in \Omega}\{ m_{\beta}(\varphi) \} \Vert \nabla\Phi \Vert_{\cL^{2}(\Omega)}^2 ,
\end{displaymath}
which is the desired result by defining, 
$l : \cW^{1,2}(\Omega) \rightarrow \R_{+},$

\begin{displaymath}
l(\varphi) := \inf_{x \in \Omega}\{ m_{\beta}(\varphi) \} .
\end{displaymath}

%
\end{proof}

Combining this result together with our early
Proposition \ref{proposition_q-convexity} yields
a new lower bound for the Bregman distance particularly 
associated with the smoothed-TV term $J_{\beta}^{\mathrm{TV}}$ 
that is formulated below.

\begin{cor}
\label{bregman_smoothed-TV_Hessian}
Under the same assumption of Theorem \ref{smoothed_TV_strongly_convex}, 
and for any $\varphi, \psi \in \cW^{1,2}(\Omega),$
the Bregman distance associated with the strongly convex 
smoothed-TV functional

\begin{displaymath}
J_{\beta}^{\mathrm{TV}}(\varphi) := \int_{\Omega} \sqrt{\norm{\nabla \varphi(x) }_2^2 + \beta} dx ,
\end{displaymath}
can be bounded below by some 
$l : \cW^{1,2}(\Omega)  \rightarrow \R_{+}$
as such

\bea
l(\varphi)\Vert \nabla\Phi \Vert_{\cL^{2}(\Omega)}^2 \norm{\varphi - \psi}_{\cL^2(\Omega)}^{2} ,
\leq D_{J_{\beta}^{\mathrm{TV}}}(\varphi, \psi) ,
\eea
where $\Phi \in \cC_{c}^{1}(\Omega) \cap \cW^{1,2}(\Omega)$ satisfies
\begin{displaymath}
(J_{\beta}^{\mathrm{TV}})^{\prime\prime}[\varphi](\Phi,\Phi) \geq l(\varphi) \norm{\nabla\Phi}_{\cL^2(\Omega)}^2 .
\end{displaymath}

\end{cor}

\begin{proof}

In the proof of Proposition \ref{proposition_q-convexity},
we set $u := \psi$ and $v := \varphi.$ This setting 
has no impact on the proof since $\norm{u - v} = \norm{v - u}.$
By this setting
and following the calculations in the regarding proof, 
and also by Theorem \ref{smoothed_TV_strongly_convex},
we associate the necessary lower bound with 
$(J_{\beta}^{\mathrm{TV}})^{\prime\prime}(\varphi)[\Phi, \Phi]$
for $\varphi \in \cW^{1,2}(\Omega).$
\end{proof}

\begin{cor}
\label{corollary_smoothed_TV_convergence}
Let the regularized solution $\varphi_{\alpha(\delta , f^{\delta})} \in \cW^{1,2}(\Omega)$
of the problem (\ref{smoothed-tv_min0}) satisfy 
Assumption \ref{assump_variational_ineq}.
Then under the same assumptions of Theorem \ref{smoothed_TV_strongly_convex}
and Corollary \ref{bregman_smoothed-TV_Hessian}, for {\em a posteriori}
rule for the choice of regularization parameter $\alpha(\delta , f^{\delta}),$

\begin{displaymath}
\norm{\varphi_{\alpha(\delta , f^{\delta})} - \varphi^{\dagger}}_{\cL^2(\Omega)}^2 \rightarrow O(\Psi(\delta))
\end{displaymath}
as $\delta \rightarrow 0$ due to the estimation (\ref{q_convexity}) 
of Definition \ref{total_convexity},
Theorem \ref{Bregman_regularizer_upper_bound2},
Theorem \ref{Bregman_regularizer_upper_bound}, and Theorem \ref{thrm_Bregman_upper_bound}.
\end{cor}

\begin{remark}
Note that the term $O(\Psi(\delta))$ in Corollary \ref{corollary_smoothed_TV_convergence}
also contains the term $\frac{1}{l(\varphi)}$ where
$l : \cW^{1,2}(\Omega) \rightarrow \R_{+}$ is defined in
Theorem \ref{smoothed_TV_strongly_convex} as well as in Corollary 
\ref{bregman_smoothed-TV_Hessian}.
\end{remark}


\section*{Acknowledgement}

The author is indepted to D. Russell Luke for the valuable
help in the formulation of Proposition \ref{proposition_q-convexity} and
is also grateful to Thorsten Hohage for the strategic discussion
both on the development of the general theory and
on the correct interpretation of the general analysis for the
smoothed-TV functional.

\newpage
\bigskip
\section*{References}


 \bibliographystyle{alpha}

\end{document}